\UseRawInputEncoding
\documentclass[12pt]{article}
\usepackage[centertags]{amsmath}
\usepackage{amsfonts}
\usepackage{amssymb}
\usepackage{latexsym}
\usepackage{amsthm}
\usepackage{newlfont}
\usepackage{graphicx}
\usepackage{listings}
\usepackage{booktabs}
\usepackage{abstract}
\usepackage{enumerate}
\usepackage{xcolor}
\RequirePackage{srcltx}
\date{}

\newlength{\defbaselineskip}
\newcommand{\setlinespacing}[1]%
           {\setlength{\baselineskip}{#1 \defbaselineskip}}

\newcommand{\actaqed}{\hfill $\actabox$}
{\medskip\noindent \textit{Proof of #1. }}%
{\actaqed \medskip}

\def\cD{{\mathcal D}}

\def\bbC{{\mathbb C}}

\def\br{\mathbf r}

\def\bt{\mathbf t}

 \def \<{\langle}
\def\>{\rangle}

\def \e{\varepsilon}

\def \ff{\varphi}

\def\bt{\beta}

\def\ga{\gamma}
\def\la{\lambda}

\def \conv{\operatorname{conv}}

\def \sp{\operatorname{span}}

\def \sign{\operatorname{sign}}

\def\bt{\beta}

\newtheorem{Theorem}{Theorem}[section]
\newtheorem{Lemma}{Lemma}[section]

\newtheorem{Remark}{Remark}[section]

\numberwithin{equation}{section}

\newcommand{\be}{\begin{equation}}
\newcommand{\ee}{\end{equation}}

\begin{document}

\title{On greedy approximation in complex Banach spaces}

\author{A. Gasnikov and V. Temlyakov}

\newcommand{\Addresses}{{
  \bigskip
  \footnotesize

\medskip
A.V. Gasnikov, \\ \textsc{Ivannikov institute for System Programming of Russian Academy of Sciences, Moscow, Russia;\\ Steklov Mathematical Institute of Russian Academy of Sciences, Moscow, Russia;  \\ Innopolis University, Tatarstan, Russia.
\\ E-mail:} \texttt{gasnikov@yandex.ru}

 \medskip
  V.N. Temlyakov, \textsc{University of South Carolina, USA,\\ Steklov Mathematical Institute of Russian Academy of Sciences, Russia;\\ Lomonosov Moscow State University, Russia; \\ Moscow Center of Fundamental and Applied Mathematics, Russia.\\  
  \\
E-mail:} \texttt{temlyakovv@gmail.com}

}}
\maketitle

\begin{abstract}{The general theory of greedy approximation with respect to arbitrary dictionaries is well developed 
  in the case of real Banach spaces. Recently, some of results proved for the Weak Chebyshev Greedy Algorithm (WCGA) in the case of real Banach spaces were extended to the case of complex Banach spaces.   In this paper  
  we extend some of known in the real case results for other than WCGA greedy algorithms  to the case of complex Banach spaces.		}
\end{abstract}

\section{Introduction}
\label{In}

Sparse and greedy approximation theory is important in applications. This theory is actively developing for about 40 years. In the beginning researchers were interested in approximation in Hilbert spaces (see, for instance,  \cite{FS}, \cite{H}, \cite{J1}, \cite{J2}, \cite{B}, \cite{DT1}, \cite{DMA}, \cite{BCDD}; for detailed history see \cite{VTbook}). Later, this theory was extended to the case of real Banach spaces (see, for instance, \cite{DGDS}, \cite{T15}, \cite{GK}, \cite{VT165}; for detailed history see \cite{VTbook} and \cite{VTbookMA}) and to the case of convex optimization (see, for instance, \cite{Cl}, \cite{FNW}, \cite{Ja2}, \cite{SSZ}, \cite{TRD}, \cite{Z}, \cite{DT}, \cite{GP}, \cite{DeTe}). The reader can find some open problems in the book \cite{VTbook}. 

We study greedy approximation with respect to arbitrary dictionaries in Banach spaces. This general theory is well developed 
  in the case of real Banach spaces (see the book \cite{VTbook}). Recently (see \cite{DGHKT}), some of results proved in the case of real Banach spaces were extended to the case of complex Banach spaces. The Weak Chebyshev Greedy Algorithm (WCGA) was studied in  \cite{DGHKT}. In this paper we discuss other than WCGA greedy algorithms. For them
  we extend some of known in the real case results  to the case of complex Banach spaces. We use the same as in the real case main ideas: First, we prove a recurrent inequality between the norms $\|f_m\|$ and $\|f_{m-1}\|$ of the residuals of the algorithm; Second, we analyze that recurrent inequality. 
The step from real case to the complex case requires some technical modifications in the proofs of the corresponding recurrent inequalities. For the reader's convenience we present a detailed analysis here. 

Let $X$ be a Banach space (real or complex) with norm $\|\cdot\|$. We say that a set of elements (functions) $\cD$ from $X$ is a dictionary if each $g\in \cD$ has norm bounded by one ($\|g\|\le1$) and the closure of $\sp \cD$ is $X$. It is convenient for us to consider along with the dictionary $\cD$ its symmetrization. In the case of real Banach spaces we 
denote 
$$
\cD^{\pm} := \{\pm g \, : \, g\in \cD\}.
$$
In the case of complex Banach spaces we denote 
$$
\cD^{\circ} := \{e^{i\theta} g \, : \, g\in \cD,\quad \theta \in [0,2\pi)\}.
$$
In the above notation $\cD^{\circ}$ symbol $\circ$ stands for the unit circle. 

Following standard notations denote 
$$
A_1^o(\cD) := \left\{f \in X\,:\, f =\sum_{j=1}^\infty a_j g_j,\quad \sum_{j=1}^\infty |a_j| \le 1, \quad g_j \in \cD,\quad j=1,2,\dots \right\}
$$
and by $A_1(\cD)$ denote the closure (in $X$) of $A_1^o(\cD)$. Then it is clear that $A_1(\cD)$ is the closure (in $X$) of the convex hull of $\cD^\pm$ in the real case and of $\cD^{\circ}$ in the complex case. Also, denote 
by $\conv(\cD)$ the closure (in $X$) of the convex hull of $\cD$.  

  We introduce a new norm, associated with a dictionary $\cD$, in the dual space $X^*$ by the formula
$$
\|F\|_\cD:=\sup_{g\in\cD}|F(g)|,\quad F\in X^*.
$$
Note, that clearly in the real case $\|F\|_\cD = \sup_{\phi\in\cD^\pm}F(\phi)$ and in the complex case $\|F\|_\cD = \sup_{\phi\in\cD^{\circ}}Re(F(\phi))$.

  In this paper we study  greedy algorithms with respect to $\cD$. 
For a nonzero element $h\in X$ we let $F_h$ denote a norming (peak) functional for $h$: 
$$
\|F_h\| =1,\qquad F_h(h) =\|h\|.
$$
The existence of such a functional is guaranteed by Hahn-Banach theorem.

 \section{Definitions and Lemmas}
\label{L}

We consider here approximation in uniformly smooth Banach spaces. For a Banach space $X$ we define the modulus of smoothness
$$
\rho(u) :=\rho(u,X) := \sup_{\|x\|=\|y\|=1}\left(\frac{1}{2}(\|x+uy\|+\|x-uy\|)-1\right).
$$
The uniformly smooth Banach space is the one with the property
$$
\lim_{u\to 0}\rho(u)/u =0.
$$
It is easy to see that for any Banach space $X$ its modulus of smoothness $\rho(u)$ is an even convex function satisfying the inequalities
$$
\max(0,u-1)\le \rho(u)\le u,\quad u\in (0,\infty).
$$
It is well known (see, for instance, \cite{DGDS})  that in the case $X=L_p$, 
 $1\le p < \infty$ we have
\be\label{Lprho}
 \rho(u,L_p) \le \begin{cases} u^p/p & \text{if}\quad 1\le p\le 2 ,\\
(p-1)u^2/2 & \text{if}\quad 2\le p<\infty. \end{cases} 
\ee 

We note that from the definition of modulus of smoothness we get the following inequality (in the real case see, for instance, \cite{VTbook}, p.336).
\begin{Lemma}\label{LL0} Let $x\neq0$. Then
\be\label{In1}
0\le \|x+uy\|-\|x\|-Re(uF_x(y))\le 2\|x\|\rho(u\|y\|/\|x\|). 
\ee
\end{Lemma}
\begin{proof} We have
$$
\|x+uy\|\ge |F_x(x+uy)|\ge Re(F_x(x+uy))=\|x\|+Re(uF_x(y)).
$$
This proves the left inequality. Next, from the definition of modulus of smoothness it follows that
\be\label{L1.7}
\|x+uy\|+\|x-uy\|\le 2\|x\|(1+\rho(u\|y\|/\|x\|)).  
\ee
Also,
\be\label{L1.8}
\|x-uy\|\ge |F_x(x-uy)| \ge Re(F_x(x-uy)) =\|x\|-Re(uF_x(y)). 
\ee
Combining (\ref{L1.7}) and (\ref{L1.8}), we obtain
$$
\|x+uy\|\le \|x\|+Re(uF_x(y))+2\|x\|\rho(u\|y\|/\|x\|).
$$
This proves the second inequality. 
\end{proof}

We will use the following two simple and well-known lemmas. In the case of real Banach spaces these lemmas are proved, for instance, in \cite{VTbook}, Ch.6, pp.342-343.
\begin{Lemma}\label{LL1} Let $X$ be a uniformly smooth Banach space and $L$ be a finite-dimensional subspace of $X$. For any $f\in X\setminus L$   let $f_L$ denote the best approximant of $f$ from $L$. Then we have 
$$
F_{f-f_L}(\phi) =0
$$
for any $\phi \in L$.
\end{Lemma}
\begin{proof} Let us assume the contrary: there is a $\phi \in L$ such that $\|\phi\|=1$ and
$$
|F_{f-f_L}(\phi)| =\bt >0.
$$
Denote by $\nu$ the complex conjugate of $\sign(F_{f-f_L}(\phi))$, where $\sign z := z/|z|$ for $z\neq 0$. Then 
$\nu F_{f-f_L}(\phi) = |F_{f-f_L}(\phi)|$.
For any $\la\ge 0$ we have from the definition of $\rho(u)$ that 
\be\label{2.3}
\|f-f_L-\la \nu\phi\| +\|f-f_L+\la \nu\phi\| \le 2\|f-f_L\|\left(1+\rho\left(\frac{\la}{\|f-f_L\|}\right)\right).
\ee
Next
\be\label{2.4}
 \|f-f_L+\la \nu\phi\| \ge |F_{f-f_L}(f-f_L+\la\nu\phi)| =\|f-f_L\|+\la\bt.
\ee
Combining (\ref{2.3}) and (\ref{2.4}) we get
\be\label{2.5}
\|f-f_L-\la\nu\phi\|\le \|f-f_L\|\left(1-\frac{\la\bt}{\|f-f_L\|} +2\rho\left(\frac{\la}{\|f-f_L\|}\right)\right).
\ee
Taking into account that $\rho(u) =o(u)$, we find $\la'>0$ such that
$$
\left(1-\frac{\la'\bt}{\|f-f_L\|} +2\rho\left(\frac{\la'}{\|f-f_L\|}\right)\right) <1.
$$
Then (\ref{2.5}) gives
$$ 
\|f-f_L-\la'\nu\phi\| < \|f-f_L\|,
$$
which contradicts the assumption that $f_L \in L$ is the best approximant of $f$.
\end{proof}

\begin{Remark}\label{SC1} The condition $F_{f-f_L}(\phi) =0$ for any $\phi \in L$ is also a sufficient condition for $f_L\in L$ 
to be a best approximant of $f$ from $L$.
\end{Remark}
\begin{proof} Indeed, for any $g \in L$ we have
$$
\|f-f_L\| = F_{f-f_L}(f-f_L) =  F_{f-f_L}(f-g) \le \|f-g\|.
$$
\end{proof}

\begin{Lemma}\label{LL2} For any bounded linear functional $F$ and any dictionary $\cD$, we have
$$
\|F\|_\cD:=\sup_{g\in \cD}|F(g)| = \sup_{f\in  A_1(\cD)} |F(f)|.
$$
\end{Lemma}
\begin{proof} The inequality 
$$
\sup_{g\in \cD}|F(g)| \le \sup_{f\in A_1(\cD)} |F(f)|
$$
is obvious because $\cD \subset A_1(\cD)$. We prove the opposite inequality. Take any $f\in A_1(\cD)$. Then for any $\e >0$ there exist $g_1^\e,\dots,g_N^\e \in \cD$ and numbers $a_1^\e,\dots,a_N^\e$ such that $|a_1^\e|+\cdots+|a_N^\e| \le 1$ and 
$$
\left\|f-\sum_{j=i}^Na_i^\e g_i^\e\right\| \le \e.
$$
Thus
$$
|F(f)| \le \|F\|\e + |F\left(\sum_{i=1}^Na_i^\e g_i^\e\right)| \le \e \|F\| +\sup_{g\in \cD} |F(g)|,
$$
which proves Lemma \ref{LL2}.
\end{proof}
 
\begin{Lemma}\label{LL3} For any bounded linear functional $F$ and any dictionary $\cD$, we have
$$
 \sup_{g\in \cD}Re(F(g)) = \sup_{f\in  \conv(\cD)} Re(F(f)).
$$
\end{Lemma}
\begin{proof} The inequality 
$$
\sup_{g\in \cD}Re(F(g)) \le \sup_{f\in \conv(\cD)} Re(F(f))
$$
is obvious because $\cD \subset \conv(\cD)$. We prove the opposite inequality. Take any $f\in \conv(\cD)$. Then for any $\e >0$ there exist $g_1^\e,\dots,g_N^\e \in \cD$ and nonnegative numbers $a_1^\e,\dots,a_N^\e$ such that $a_1^\e+\cdots+a_N^\e = 1$ and 
$$
\left\|f-\sum_{j=i}^Na_i^\e g_i^\e\right\| \le \e.
$$
Thus
$$
Re(F(f)) \le \|F\|\e + Re(F\left(\sum_{i=1}^Na_i^\e g_i^\e\right)) \le \e \|F\| +\sup_{g\in \cD} Re(F(g)),
$$
which proves Lemma \ref{LL3}.
\end{proof}

\section{Main results}
\label{M}

In this section we discuss three greedy type algorithms and prove convergence  and rate of convergence results for those algorithms. In the case of real Banach spaces these algorithms and the corresponding results on their convergence  and rate of convergence are known. We show here how to modify those algorithms to the case of complex Banach spaces. 

\subsection{WGAFR} 
\label{MSub1}

The following algorithm -- WGAFR -- can be defined in the same way in both the real and the complex cases. We present the definition in the complex case. 

{\bf Weak Greedy Algorithm with Free Relaxation  (WGAFR).} \newline
Let $\tau:=\{t_m\}_{m=1}^\infty$, $t_m\in[0,1]$, be a weakness  sequence. We define $f_0   :=f$ and $G_0  := 0$. Then for each $m\ge 1$ we have the following inductive definition.

(1) $\varphi_m   \in \cD$ is any element satisfying
$$
|F_{f_{m-1}}(\varphi_m)| \ge t_m  \|F_{f_{m-1}}\|_\cD.
$$

(2) Find $w_m\in \bbC$ and $ \lambda_m\in\bbC$ such that
$$
\|f-((1-w_m)G_{m-1} + \la_m\varphi_m)\| = \inf_{ \la,w\in\bbC}\|f-((1-w)G_{m-1} + \la\varphi_m)\|
$$
and define
$$
G_m:=   (1-w_m)G_{m-1} + \la_m\varphi_m.
$$

(3) Let
$$
f_m   := f-G_m.
$$

We begin with the main lemma. In the case of real Banach spaces this lemma is known (see, for instance, \cite{VTbook}, p. 350, Lemma 6.20). 
 \begin{Lemma}\label{ML1} Let $X$ be a uniformly smooth complex Banach space with modulus of smoothness $\rho(u)$. Take a number $\e\ge 0$ and two elements $f$, $f^\e$ from $X$ such that
$$
\|f-f^\e\| \le \e,\quad
f^\e/A(\e) \in A_1(\cD),
$$
with some number $A(\e)\ge \e$.
Then we have for the residual of the WGAFR after $m$ iterations
$$
\|f_m\| \le \|f_{m-1}\|\inf_{\la\ge0}\left(1-\la t_mA(\e)^{-1}\left(1-\frac{\e}{\|f_{m-1}\|}\right)
+ 2\rho\left(\frac{5\la}{\|f_{m-1}\|}\right)\right),
$$
for $\quad m=1,2,\dots .$
\end{Lemma}
\begin{proof} Denote by $\nu$ the complex conjugate of $\sign F_{f_{m-1}}(\varphi_m)$, where $\sign z := z/|z|$ for $z\neq 0$. Then $\nu F_{f_{m-1}}(\varphi_m) = |F_{f_{m-1}}(\varphi_m)|$.
By the definition of $f_m$
$$
\|f_m\|\le \inf_{\la\ge0,w\in\bbC}\|f_{m-1}+wG_{m-1}-\la\nu\ff_m\|.
$$
As in the argument in the proof of Lemma \ref{LL1} we
use the inequality 
$$
\|f_{m-1}+wG_{m-1}-\la\nu\ff_m\|+\|f_{m-1}-wG_{m-1}+\la\nu\ff_m\|\le  
$$
\be\label{4.1}
2\|f_{m-1}\|(1+\rho(\|wG_{m-1}-\la\nu\ff_m\|/\|f_{m-1}\|))
\ee
and estimate for $\la\ge0$
$$
\|f_{m-1}-wG_{m-1}+\la\nu\ff_m\|\ge Re(F_{f_{m-1}}(f_{m-1}-wG_{m-1}+\la\nu\ff_m))\ge
$$
$$
\|f_{m-1}\|-Re(F_{f_{m-1}}(wG_{m-1}))+\la t_m\sup_{g\in\cD}|F_{f_{m-1}}(g)|=
$$
By Lemma \ref{LL2}, we continue:  
$$
\|f_{m-1}\|-Re(F_{f_{m-1}}(wG_{m-1}))+\la t_m\sup_{\phi\in A_1(\cD)}|F_{f_{m-1}}(\phi)|\ge
$$
$$
\|f_{m-1}\|-Re(F_{f_{m-1}}(wG_{m-1}))+\la t_m A(\e)^{-1}|F_{f_{m-1}}(f^\e)|\ge
$$
$$
\|f_{m-1}\|-Re(F_{f_{m-1}}(wG_{m-1}))+\la t_m A(\e)^{-1}(Re(F_{f_{m-1}}(f))-\e).
$$
We set $w^*:=\la t_mA(\e)^{-1}$ and obtain
\be\label{4.2}
\|f_{m-1}-w^*G_{m-1}+\la\nu\ff_m\|\ge\|f_{m-1}\|+\la t_mA(\e)^{-1}(\|f_{m-1}\|-\e). 
\ee
Combining (\ref{4.1}) and (\ref{4.2}) we get 
$$
\|f_m\| \le \|f_{m-1}\|\inf_{\la\ge0}(1-\la t_mA(\e)^{-1}(1-\e/\|f_{m-1}\|) 
$$
$$
 + 2\rho (\|w^*G_{m-1}-\la\nu\ff_m\|/\|f_{m-1}\|)).
$$
We now estimate
$$
\|w^*G_{m-1}-\la\nu\ff_m\| \le w^*\|G_{m-1}\|+\la.
$$
Next,
$$
\|G_{m-1}\|=\|f-f_{m-1}\|\le 2\|f\|\le 2(\|f^\e\|+\e)\le2(A(\e)+\e).
$$
Thus, under assumption $A(\e)\ge\e$ we get
$$
w^*\|G_{m-1}\|\le 2\la t_m(A(\e)+\e)/A(\e) \le 4\la.
$$
Finally,
$$
\|w^*G_{m-1}-\la\ff_m\|\le 5\la.
$$
This completes the proof of Lemma \ref{ML1}.
\end{proof}

\begin{Remark}\label{MR1} It follows from the definition of the WGAFR that 
the sequence $\{\|f_m\|\}$ is a non-icreasing sequence. 
\end{Remark}

We now prove a convergence theorem for an arbitrary uniformly smooth Banach space. Modulus of smoothness $\rho(u)$ of a uniformly smooth Banach space is an even convex function such that $\rho(0)=0$ and  $\lim_{u\to0}\rho(u)/u=0$. The  function $s(u):=\rho(u)/u$, $s(0):=0$, associated with $\rho(u)$ is a continuous increasing function on $[0,\infty)$. Therefore, the inverse function $s^{-1}(\cdot)$ is well defined. The following Theorem \ref{MT1}  is known  in the case of real Banach spaces (see, for instance, \cite{VTbook}, p. 352, Theorem 6.22).

\begin{Theorem}\label{MT1} Let $X$ be a uniformly smooth complex Banach space with modulus of smoothness $\rho(u)$. Assume that a sequence $\tau :=\{t_k\}_{k=1}^\infty$ satisfies the following condition. For any $\theta >0$ we have
\be\label{4.3}
\sum_{m=1}^\infty t_m  s^{-1}(\theta t_m)=\infty.  
\ee
 Then, for any $f\in X$ we have for the WGAFR
$$
\lim_{m\to \infty} \|f_m\| =0.
$$
\end{Theorem} 
\begin{proof}  This proof only uses Lemma \ref{ML1}, which provides relation between the residuals $\|f_m\|$ and $\|f_{m-1}\|$. It repeats the corresponding proof in the real case. For completeness we present this simple proof here. 
By Remark \ref{MR1}, $\{\|f_m\|\}$ is a non-increasing sequence. Therefore we have
$$
\lim_{m\to \infty}\|f_m\| =\bt.
$$
We prove that $\bt =0$ by contradiction. Assume the contrary, that $\bt>0$. Then, for any $m$ we have
$$
\|f_m\| \ge \bt.
$$
We set $\e =\bt/2$ and find $f^\e$ such that
$$
\|f-f^\e\| \le \e \quad \text{and}\quad f^\e/A(\e) \in A_1(\cD),
$$
with some $A(\e)\ge\e$. Then, by Lemma \ref{ML1} we get
$$
\|f_m\|  \le \|f_{m-1}\|\inf_{\la\ge0} (1-\la t_mA(\e)^{-1}/2    +2\rho(5\la/\bt)).
$$
Let us specify $\theta:=\bt/(40A(\e))$ and take $\la = \bt s^{-1}(\theta t_m)/5$. Then we obtain
$$
\|f_m\| \le \|f_{m-1}\|(1-2\theta t_m s^{-1}(\theta t_m)).
$$
The assumption
$$
\sum_{m=1}^\infty t_m s^{-1}(\theta t_m) =\infty
$$
implies that
$$
\|f_m\| \to 0 \quad \text{as} \quad m\to \infty.
$$
We have a contradiction, which proves the theorem.
\end{proof} 

\begin{Remark}\label{NC1} The reader can find some necessary conditions on the weakness sequence $\tau$ for 
convergence of the WCGA, for instance, in \cite{VTbook}, Proposition 6.13, p. 346.
\end{Remark}

We now proceed to the rate of convergence results. The following Theorem \ref{MT2}  is known  in the case of real Banach spaces (see, for instance, \cite{VTbook}, p. 353, Theorem 6.23).

\begin{Theorem}\label{MT2} Let $X$ be a uniformly smooth Banach space with modulus of smoothness $\rho(u)\le \gamma u^q$, $1<q\le 2$. Take a number $\e\ge 0$ and two elements $f$, $f^\e$ from $X$ such that
$$
\|f-f^\e\| \le \e,\quad
f^\e/A(\e) \in A_1(\cD),
$$
with some number $A(\e)>0$.
Then we have for the residual of the WGAFR after $m$ iterations
$$
\|f_m\| \le  \max\left(2\e, C(q,\gamma)(A(\e)+\e)\left(1+\sum_{k=1}^mt_k^p\right)^{-1/p}\right),\quad p:=q/(q-1). 
$$
\end{Theorem}
\begin{proof} This proof only uses Lemma \ref{ML1}, which provides relation between the residuals $\|f_m\|$ and $\|f_{m-1}\|$. It repeats the corresponding proof in the real case. For completeness we present this proof here.  It is clear that it suffices to consider the case $A(\e)\ge\e$. Otherwise, $\|f_m\|\le\|f\|\le\|f^\e\|+\e\le2\e$. 
Also, assume $\|f_m\|>2\e$ (otherwise Theorem \ref{MT2} trivially holds). Then, by Remark \ref{MR1} we have for all $k=0,1,\dots,m$ that $\|f_k\|>2\e$. 
By Lemma \ref{ML1} we obtain
\be\label{4.4}
\|f_k\| \le \|f_{k-1}\|\inf_{\la\ge0} \left(1-\la t_kA(\e)^{-1}/2 + 2\gamma\left(\frac{5\la}{\|f_{k-1}\|}\right)^q\right).
\ee
Choose $\la$ from the equation
$$
\frac{\la t_k}{4A(\e)} = 2\gamma \left(\frac{5\la}{\|f_{k-1}\|}\right)^q,
$$
which implies that
$$
\la = \|f_{k-1}\|^{\frac{q}{q-1}}5^{-\frac{q}{q-1}}(8\gamma A(\e))^{-\frac{1}{q-1}}t_k^{\frac{1}{q-1}}.
$$
Define 
$$
A_q := 4(8\gamma)^{\frac{1}{q-1}}5^{\frac{q}{q-1}}.
$$
Using notation $p:= \frac{q}{q-1}$, we get from (\ref{4.4})
$$
\|f_k\| \le \|f_{k-1}\|\left(1-\frac{1}{4}\frac{\la t_k}{A(\e)}\right) = \|f_{k-1}\|\left(1-\frac{t_k^p\|f_{k-1}\|^p}{A_qA(\e)^p}\right).
$$
Raising both sides of this inequality to the power $p$ and taking into account the inequality $x^r\le x$ for $r\ge 1$, $0\le x\le 1$, we obtain
\be\label{M4}
\|f_k\|^p \le \|f_{k-1}\|^p \left(1-\frac{t^p_k\|f_{k-1}\|^p}{A_qA(\e)^p}\right).
\ee
We shall need the
following simple lemma, different versions of which are well known (see, for example, \cite{VTbook}, p. 91). 

\begin{Lemma}\label{HL1} Let a number $C_1>0$ and a sequence $\{a_k\}_{k=1}^\infty$, $a_k >0$, $k=1,2,\dots$, be given.
Assume that $\{x_m\}_{m=0}^\infty$
is a sequence of non-negative
 numbers satisfying the inequalities
$$
x_0 \le C_1, \quad x_{m+1} \le x_m(1 - x_m a_{m+1}) , \quad m = 1,2, \dots,\quad C_1,C_2>0 .
$$
Then we have for each $m$
$$
x_m \le \left(C_1^{-1}+\sum_{k=1}^{m} a_k\right)^{-1} .
$$
\end{Lemma}
\begin{proof} The proof is by induction on $m$. For $m = 0$ the statement
is true by assumption. We prove that
 $$
 x_m \le \left(C_1^{-1}+\sum_{k=1}^{m} a_k\right)^{-1}\quad \text{implies} \quad   x_{m+1} \le
\left(C_1^{-1}+\sum_{k=1}^{m+1} a_k\right)^{-1}.
$$
 If $x_{m+1} = 0$ this statement is obvious. Assume therefore
that $x_{m+1} > 0$. Then we have 
$$
x_{m+1}^{-1} \ge x_m^{-1}(1 - x_m a_{m+1})^{-1} \ge x_m^{-1}(1 + x_m a_{m+1}) =
x_m^{-1} + a_{m+1} \ge C_1^{-1}+\sum_{k=1}^{m+1} a_k ,
$$
which proves the required inequality. 
\end{proof}

We use Lemma \ref{HL1} with $x_k := \|f_k\|^p$. Using the estimate $\|f\| \le A(\e)+\e$, we set $C_1:= A(\e)+\e$. 
We specify $a_k := t^p_k(A_qA(\e)^p)^{-1}$. Then (\ref{M4}) guarantees that we can apply Lemma \ref{HL1}. Note that $A_q>1$. Then Lemma \ref{HL1} gives
\be\label{M5}
\|f_m\|^p \le A_q(A(\e)+\e)^p\left(1+\sum_{k=1}^m t_k^p\right)^{-1},
\ee
which implies
$$
\|f_m\|\le C(q,\gamma)(A(\e)+\e)\left(1+\sum_{k=1}^m t_k^p\right)^{-1/p}.
$$
Theorem \ref{MT2} is proved.
\end{proof}

\subsection{GAWR}
\label{MSub2}

In this subsection we study the  Greedy Algorithm with Weakness parameter $t$ and Relaxation $\br$ (GAWR($t,\br$)),
$\br := \{r_j\}_{j=1}^\infty$, $r_j\in [0,1)$, $j=1,2,\dots$. In addition to the acronym  GAWR($t,\br$) we will use the abbreviated acronym GAWR for the name of this algorithm. We give a general definition of the algorithm in the case of a weakness sequence $\tau$.

{\bf Greedy Algorithm with Weakness and Relaxation GAWR($\tau,\br$).} 
Let $\tau:=\{t_m\}_{m=1}^\infty$, $t_m\in[0,1]$, be a weakness  sequence. We define $f_0   :=f$ and $G_0  := 0$. Then for each $m\ge 1$ we inductively define

1). $\varphi_m   \in \cD$ is any satisfying
$$
|F_{f_{m-1}}(\varphi_m)| \ge t_m  \| F_{f_{m-1}}\|_\cD.
$$

2). Find $ \lambda_m \in \bbC$ such that
$$
\|f-((1-r_m)G_{m-1} + \la_m\varphi_m)\| = \inf_{ \la\in \bbC}\|f-((1-r_m)G_{m-1} + \la\varphi_m)\|
$$
and define
$$
G_m:=   (1-r_m)G_{m-1} + \la_m\varphi_m.
$$

3). Denote
$$
f_m   := f-G_m.
$$

We now prove convergence and rate of convergence results for the \newline
GAWR($t,\br$), i.e. for the GAWR($\tau,\br$) with $\tau =\{t_j\}_{j=1}^\infty$, $t_j=t \in (0,1]$, $j=1,2,\dots$.   We begin with an analogue of Lemma \ref{ML1}. In the case of real Banach spaces Lemma \ref{ML3} is known (see, for instance, \cite{VTbook}, p. 355, Lemma 6.24).

\begin{Lemma}\label{ML3} Let $X$ be a uniformly smooth complex Banach space with modulus of smoothness $\rho(u)$. Take a number $\e\ge 0$ and two elements $f$, $f^\e$ from $X$ such that
$$
\|f-f^\e\| \le \e,\quad
f^\e/A(\e) \in A_1(\cD),
$$
with some number $A(\e)>0$.
Then we have for the GAWR($t,\br$) 
$$
\|f_m\| \le \|f_{m-1}\|\left(1- r_m\left(1-\frac{\e}{\|f_{m-1}\|}\right) + 2\rho\left(\frac{r_m(\|f\|+A(\e)/t)}{(1-r_m)\|f_{m-1}\|}\right)\right),
$$
 for $m=1,2,\dots $.
\end{Lemma} 
\begin{proof} 
As above in the proof of Lemma \ref{ML1}, denote by $\nu$ the complex conjugate of $\sign F_{f_{m-1}}(\varphi_m)$, where $\sign z := z/|z|$ for $z\neq 0$. Then $\nu F_{f_{m-1}}(\varphi_m) = |F_{f_{m-1}}(\varphi_m)|$.
By the definition of $f_m$ 
$$
\|f_m\|\le \inf_{\la\ge0}\|f-((1-r_m)G_{m-1}+\la\nu\ff_m)\|.
$$
We have for any $\lambda$
\be\label{M2.3}
f-((1-r_m)G_{m-1}+\la\nu \varphi_m) = (1-r_m)f_{m-1}+r_mf-\la\nu \ff_m  
\ee
and from the definition of $\rho(u)$ with $u=\frac{\|r_mf- \la\nu\ff_m\|}{(1-r_m)\|f_{m-1}\|}$ obtain
$$
\|(1-r_m)f_{m-1}+r_mf-\la\nu \ff_m\|+\|(1-r_m)f_{m-1}-r_mf+\la\nu \ff_m\|\le  
$$
\be\label{M2.4}
2(1-r_m)\|f_{m-1}\|\left(1+\rho\left(\frac{\|r_mf- \la\nu\ff_m\|}{(1-r_m)\|f_{m-1}\|}\right)\right).
\ee
We have  
$$
\|(1-r_m)f_{m-1}-r_mf+\la \nu\ff_m\|\ge Re(F_{f_{m-1}}((1-r_m)f_{m-1}-r_mf+\la\nu \ff_m))= 
$$
\be\label{M2.5}
(1-r_m)\|f_{m-1}\|-r_mRe(F_{f_{m-1}}(f))+\la Re(\nu F_{f_{m-1}}(\ff_m)).  
\ee
From the definition of $\ff_m$ we get
\be\label{M2.6}
\nu F_{f_{m-1}}(\ff_m) = |F_{f_{m-1}}(\ff_m)| \ge t\sup_{g\in\cD}|F_{f_{m-1}}(g)|. 
\ee
By Lemma \ref{LL2}   we obtain
$$
\sup_{g\in\cD} |F_{f_{m-1}}(g)| = \sup_{\phi\in A_1(\cD)}|F_{f_{m-1}}(\phi)| 
$$
\be\label{M2.7}
 \ge 
A(\e)^{-1}|F_{f_{m-1}}(f^\e)|\ge A(\e)^{-1}(|F_{f_{m-1}}(f)|-\e).  
\ee
Combining (\ref{M2.6}) and (\ref{M2.7}) we get
\be\label{M2.8}
|F_{f_{m-1}}(\ff_m)|\ge tA(\e)^{-1}(|F_{f_{m-1}}(f)|-\e).  
\ee
We now choose $\la:=\la^*:=r_mA(\e)/t$. Then for this $\la$ we derive from (\ref{M2.5}) and (\ref{M2.8})
\be\label{M2.9}
\|(1-r_m)f_{m-1}-r_mf+\la^*\nu \ff_m\|\ge (1-r_m)\|f_{m-1}\|-r_m\e.  
\ee
The relations (\ref{M2.4}) and (\ref{M2.9}) imply
$$
\|(1-r_m)f_{m-1}+r_mf-\la^*\nu \ff_m\|\le 
$$
$$
(1-r_m)\|f_{m-1}\|+r_m\e+
2(1-r_m)\|f_{m-1}\|\rho\left(\frac{r_m(\|f\|+A(\e)/t)}{(1-r_m)\|f_{m-1}\|}\right). 
$$
\end{proof}

We begin with a convergence result. It is known that the version of Lemma \ref{ML3} for the real case implies the corresponding convergence result -- Theorem \ref{MT3} below -- in the real case (see, for instance, \cite{VTbook}, pp. 355--357). That same proof derives Theorem \ref{MT3} from Lemma \ref{ML3}. We do not present it here. 

 \begin{Theorem}\label{MT3} Let a sequence $\br:=\{r_k\}_{k=1}^\infty$, $r_k\in[0,1)$, satisfy the conditions
$$
\sum_{k=1}^\infty r_k =\infty,\quad r_k\to 0\quad\text{as}\quad k\to\infty.
$$
Then the GAWR($t,\br$)   converges in any uniformly smooth Banach space for each $f\in X$ and for all dictionaries $\cD$.
\end{Theorem}

The situation with the rate of convergence results is very similar to the above described situation with the convergence result. The real case proof (see, for instance, \cite{VTbook}, pp. 357--358) derives Theorem \ref{MT4} from Lemma \ref{ML3}. We do not present it here. 

\begin{Theorem}\label{MT4} Let $X$ be a uniformly smooth Banach space with modulus of smoothness $\rho(u)\le \gamma u^q$, $1<q\le 2$. Let 
$\br:=\{2/(k+2)\}_{k=1}^\infty$. Consider the GAWR($t,\br$). For a pair of functions $f$, $f^\e$, satisfying
$$
\|f-f^\e\|\le \e,\quad f^\e/A(\e)\in A_1(\cD)
$$
we have
$$
\|f_m\|\le \e+C(q,\gamma)(\|f\|+A(\e)/t)m^{-1+1/q}.
$$
\end{Theorem}

\subsection{Incremental Algorithm}
\label{MSub3}

We proceed to  one more greedy type algorithm (see \cite{VT94}).  
Let $\e=\{\e_n\}_{n=1}^\infty $, $\e_n> 0$, $n=1,2,\dots$. The following greedy algorithm was introduced and 
studied in \cite{VT94} (see also, \cite{VTbook}, pp. 361--363) in the case of real Banach spaces.

 {\bf Incremental Algorithm with schedule $\e$ (IA($\e$)).} 
Let $f\in \conv(\cD)$. Denote $f_0^{i,\e}:= f$ and $G_0^{i,\e} :=0$. Then, for each $m\ge 1$ we have the following inductive definition.

(1) $\ff_m^{i,\e} \in \cD$ is any element satisfying
$$
F_{f_{m-1}^{i,\e}}(\ff_m^{i,\e}-f) \ge -\e_m.
$$

(2) Define
$$
G_m^{i,\e}:= (1-1/m)G_{m-1}^{i,\e} +\ff_m^{i,\e}/m.
$$

(3) Let
$$
f_m^{i,\e} := f- G_m^{i,\e}.
$$
 
It is clear from the definition of the {\bf IA($\e$)} that 
\be\label{M14}
G_m^{i,\e} = \frac{1}{m} \sum_{j=1}^m \ff_j^{i,\e}. 
\ee
 
 We now modify the definition of {\bf IA($\e$)} to make it suitable for the complex Banach spaces. 
 
 {\bf Incremental Algorithm (complex) with schedule $\e$ (IAc($\e$)).} 
Let $f\in A_1(\cD)$. Denote $f_0^{c,\e}:= f$ and $G_0^{c,\e} :=0$. Then, for each $m\ge 1$ we have the following inductive definition.

(1) $\ff_m^{c,\e} \in \cD^\circ$ is any element satisfying
$$
Re(F_{f_{m-1}^{c,\e}}(\ff_m^{c,\e}-f)) \ge -\e_m.
$$
Denote by $\nu_m$ the complex conjugate of $\sign F_{f_{m-1}^{c,\e}}(\ff_m^{c,\e})$, where $\sign z := z/|z|$ for $z\neq 0$.

(2) Define
$$
G_m^{c,\e}:= (1-1/m)G_{m-1}^{c,\e} +\nu_m\ff_m^{c,\e}/m.
$$

(3) Let
$$
f_m^{c,\e} := f- G_m^{c,\e}.
$$

It follows from the definition of the {\bf IAc($\e$)} that 
\be\label{M15}
G_m^{c,\e} = \frac{1}{m} \sum_{j=1}^m \nu_j\ff_j^{c,\e},\quad |\nu_j|=1, \quad j=1,2,\dots. 
\ee

Note that by the definition of $\nu_m$ we have $F_{f_{m-1}^{c,\e}}(\nu_m\ff_m^{c,\e}) = |F_{f_{m-1}^{c,\e}}(\ff_m^{c,\e})|$ 
and, therefore, by Lemma \ref{LL2}
$$
\sup_{\phi \in \cD^\circ}Re(F_{f_{m-1}^{c,\e}}(\phi))= \sup_{g \in \cD}|F_{f_{m-1}^{c,\e}}(g)| 
$$
$$
= \sup_{\phi \in A_1(\cD)}|F_{f_{m-1}^{c,\e}}(\phi)| \ge |F_{f_{m-1}^{c,\e}}(f)| \ge Re(F_{f_{m-1}^{c,\e}}(f)).
$$
This means that we can always run the IAc($\e$) for $f\in A_1(\cD)$.
 
\begin{Theorem}\label{MT5} Let $X$ be a uniformly smooth complex Banach space with  modulus of smoothness $\rho(u)\le \gamma u^q$, $1<q\le 2$. Define
$$
\e_n := K_1\gamma ^{1/q}n^{-1/p},\qquad p=\frac{q}{q-1},\quad n=1,2,\dots .
$$
Then, for any $f\in  A_1(\cD)$ we have
$$
\|f_m^{c,\e}\| \le C(K_1) \gamma^{1/q}m^{-1/p},\qquad m=1,2\dots.
$$
\end{Theorem}
\begin{proof} We will use the abbreviated notation $f_m:=f_m^{c,\e}$, $\ff_m:=\nu_m\ff_m^{c,\e}$, $G_m:=G_m^{c,\e}$. Writing
$$
f_m = f_{m-1}-(\ff_m-G_{m-1})/m,
$$
we  immediately obtain the trivial estimate
\be\label{M6.6}
\|f_m\|\le \|f_{m-1}\| +2/m.  
\ee
Since
\be\label{M6.7}
f_m=\left(1-\frac{1}{m}\right)f_{m-1}-\frac{\ff_m-f}{m} 
 = \left(1-\frac{1}{m}\right)\left(f_{m-1}-\frac{\ff_m-f}{m-1}\right) 
\ee
we obtain by Lemma \ref{LL0} with $y= \ff_m-f$ and $u= -1/(m-1)$
$$
\left\|f_{m-1} -\frac{\ff_m-f}{m-1}\right\| 
$$
\be\label{M6.8}
\le \|f_{m-1}\|\left(1+2\rho\left(\frac{2}{(m-1)\|f_{m-1}\|}\right)\right) +\frac{\e_m}{m-1}.  
\ee
 Using the definition of $\e_m$ and the assumption $\rho(u) \le \gamma u^q$, we make the following observation. There exists a constant $C(K_1)$ such that, if
\be\label{M6.9}
\|f_{m-1}\|\ge C(K_1) \gamma ^{1/q}(m-1)^{-1/p}  
\ee
then
\be\label{M6.10}
2\rho(2((m-1)\|f_{m-1}\|)^{-1}) + \e_m ((m-1)\|f_{m-1}\|)^{-1} \le 1/(4m),  
\ee
and therefore, by (\ref{M6.7}) and (\ref{M6.8})
\be\label{M6.11}
\|f_m\| \le \left(1-\frac{3}{4m}\right)\|f_{m-1}\|.  
\ee

We now need the following   technical lemma (see, for instance, \cite{VTbook}, p.357).
\begin{Lemma}\label{ML4}  Let a sequence $\{a_n\}_{n=1}^\infty$ have the following property. For given positive
 numbers $\alpha < \ga \le 1$, $A > a_1$,  we have for all $n\ge2$
 \be\label{M5.6}
 a_n\le a_{n-1}+A(n-1)^{-\alpha}.  
 \ee
 If for some $v \ge2$ we have
$$
a_v \ge Av^{-\alpha}
$$
then
 \be\label{M5.7}
a_{v + 1} \le a_v (1- \ga/v). 
\ee
Then there exists a constant $C(\alpha , \ga )$ such that for all $n=1,2,\dots $ we have
$$
a_n \le C(\alpha,\ga)An^{-\alpha} .
$$
 \end{Lemma}

Taking into account (\ref{M6.6}) we apply Lemma \ref{ML4}  to the sequence $a_n=\|f_n\|$, $n=1,2,\dots$ with $\alpha=1/p$, $\beta =3/4$ and complete the proof of Theorem \ref{MT5}.
\end{proof}

We now consider one more modification  of {\bf IA($\e$)} suitable for the complex Banach spaces and providing the convex combination of the dictionary elements. In the case of real Banach spaces it coincides with the {\bf IA($\e$)}.
 
 {\bf Incremental Algorithm (complex and convex) with schedule $\e$ (IAcc($\e$)).} 
Let $f\in \conv(\cD)$. Denote $f_0^{cc,\e}:= f$ and $G_0^{cc,\e} :=0$. Then, for each $m\ge 1$ we have the following inductive definition.

(1) $\ff_m^{cc,\e} \in \cD$ is any element satisfying
$$
Re(F_{f_{m-1}^{cc,\e}}(\ff_m^{cc,\e}-f)) \ge -\e_m.
$$
 
(2) Define
$$
G_m^{cc,\e}:= (1-1/m)G_{m-1}^{cc,\e} + \ff_m^{cc,\e}/m.
$$

(3) Let
$$
f_m^{cc,\e} := f- G_m^{cc,\e}.
$$

It follows from the definition of the {\bf IAcc($\e$)} that 
\be\label{M15}
G_m^{cc,\e} = \frac{1}{m} \sum_{j=1}^m \ff_j^{cc,\e} . 
\ee

Note that by   Lemma \ref{LL3}
$$
 \sup_{g \in \cD}Re(F_{f_{m-1}^{cc,\e}}(g))
= \sup_{\phi \in \conv(\cD)}Re(F_{f_{m-1}^{cc,\e}}(\phi))   \ge Re(F_{f_{m-1}^{cc,\e}}(f)).
$$
This means that we can always run the {\bf IAcc($\e$)} for $f\in \conv(\cD)$.
 
In the same way as Theorem \ref{MT5} was proved above one can prove the following Theorem \ref{MT6}.

\begin{Theorem}\label{MT6} Let $X$ be a uniformly smooth Banach space with  modulus of smoothness $\rho(u)\le \gamma u^q$, $1<q\le 2$. Define
$$
\e_n := K_1\gamma ^{1/q}n^{-1/p},\qquad p=\frac{q}{q-1},\quad n=1,2,\dots .
$$
Then, for any $f\in  \conv(\cD)$ we have
$$
\|f_m^{cc,\e}\| \le C(K_1) \gamma^{1/q}m^{-1/p},\qquad m=1,2\dots.
$$
\end{Theorem}

{\bf Acknowledgements.}
This work was supported by Ministry of Science and Higher Education of 
the Russian Federation (Grant No. 075-15-2024-529).

 \Addresses
 
\end{document}